\documentclass[11pt,oneside,english,twoside,reqno]{amsart}
\usepackage{ae,aecompl}
\usepackage[T1]{fontenc}
\usepackage[latin9]{inputenc}
\usepackage[letterpaper]{geometry}
\usepackage{babel}
\usepackage{enumitem}
\usepackage{amstext}
\usepackage{amsthm}
\usepackage{amssymb}
\usepackage{stmaryrd}
\usepackage{setspace}
\usepackage[unicode=true,
 bookmarks=true,bookmarksnumbered=false,bookmarksopen=true,bookmarksopenlevel=2,
 breaklinks=false,pdfborder={0 0 1},backref=false,colorlinks=false]
 {hyperref}

\makeatletter
\numberwithin{equation}{section}
\theoremstyle{plain}
\newtheorem{thm}{\protect\theoremname}[section]
\theoremstyle{plain}
\newtheorem{conjecture}[thm]{\protect\conjecturename}
\theoremstyle{plain}
\newtheorem{cor}[thm]{\protect\corollaryname}
\theoremstyle{remark}
\newtheorem{rem}[thm]{\protect\remarkname}
\theoremstyle{plain}
\newtheorem{lem}[thm]{\protect\lemmaname}
\theoremstyle{plain}
\newtheorem{prop}[thm]{\protect\propositionname}

\allowdisplaybreaks

\setlist[enumerate]{label={\upshape(\arabic*)},leftmargin=*}
\setlist[itemize]{leftmargin=*}

\newtheorem{ConditionAB}{Condition}

\makeatother

\providecommand{\conjecturename}{Conjecture}
\providecommand{\corollaryname}{Corollary}
\providecommand{\lemmaname}{Lemma}
\providecommand{\propositionname}{Proposition}
\providecommand{\remarkname}{Remark}
\providecommand{\theoremname}{Theorem}

\begin{document}
\title{B\"ottcher coordinates at wild superattracting fixed points}
\author{Hang Fu}
\address{IAZD, Leibniz Universit\"at Hannover, Welfengarten 1, 30167 Hannover,
Germany}
\email{drfuhang@gmail.com}
\urladdr{https://sites.google.com/view/hangfu}
\author{Hongming Nie}
\address{Institute for Mathematical Sciences, Stony Brook University}
\email{hongming.nie@stonybrook.edu}
\urladdr{https://sites.google.com/view/hmnie}
\date{\today}
\begin{abstract}
Let $p$ be a prime number, let $g(x)=x^{p^{2}}+p^{r+2}x^{p^{2}+1}$
with $r\in\mathbb{Z}_{\geq0}$, and let $\phi(x)=x+O(x^{2})$ be the
B\"ottcher coordinate satisfying $\phi(g(x))=\phi(x)^{p^{2}}$. Salerno
and Silverman conjectured that the radius of convergence of $\phi^{-1}(x)$
in $\mathbb{C}_{p}$ is $p^{-p^{-r}/(p-1)}$. In this article, we
confirm that this conjecture is true by showing that it is a special
case of our more general result.
\end{abstract}

\maketitle

\section{\label{sec1} Introduction}

Let $K$ be a field of characteristic $0$ and let $g(x)=x^{d}+O(x^{d+1})\in K\left\llbracket x\right\rrbracket $
with $d\geq2$. Then there is a unique \textbf{B\"ottcher coordinate}
$\phi(x)=x+O(x^{2})\in K\left\llbracket x\right\rrbracket $ satisfying
$\phi(g(x))=\phi(x)^{d}$. It can be seen that
\[
\phi(x)=\lim_{n\to\infty}g^{n}(x)^{-1/d^{n}}.
\]

While the B\"ottcher coordinate over $K=\mathbb{C}$ has become a
fundamental tool in the area of complex dynamics (see, for example,
\cite[Chapter 9]{Mil06} for more details), its analogue over $K=\mathbb{C}_{p}$
has only been studied from the last decade. Ingram \cite{Ing13} used
$p$-adic B\"ottcher coordinates to study arboreal Galois representations.
DeMarco et al. \cite{DGKNTY19} used $p$-adic B\"ottcher coordinates
to prove a theorem of unlikely intersections. Salerno and Silverman
\cite{SS20} studied the integrality properties of some $p$-adic
B\"ottcher coordinates. In particular, they proposed the following
conjecture \cite[Conjecture 27]{SS20}.
\begin{conjecture}
[Salerno and Silverman] \label{conj1.1} Let $p$ be a prime number,
let
\[
g(x)=x^{p^{2}}+p^{r+2}x^{p^{2}+1}
\]
with $r\in\mathbb{Z}_{\geq0}$, and let $\phi(x)=x+O(x^{2})$ be the
B\"ottcher coordinate satisfying $\phi(g(x))=\phi(x)^{p^{2}}$. Then
the radius of convergence of $\phi^{-1}(x)$ in $\mathbb{C}_{p}$
is $p^{-p^{-r}/(p-1)}$.
\end{conjecture}

In this article, we will prove a generalization of Conjecture \ref{conj1.1}.
Before stating our main results, we first briefly explain how we approach
the solution of this problem.

Let $f_{c}(z)=z^{d}-c$ for some $c\in\mathbb{C}_{p}$ and let
\begin{equation}
\varphi_{c}(z)=z\left(1+\sum_{n=1}^{\infty}\frac{a_{n}}{z^{nd}}\right)\label{eq1.1}
\end{equation}
satisfy the functional equation
\begin{equation}
f_{c}(\varphi_{c}(z))=\varphi_{c}(z^{d}).\label{eq1.2}
\end{equation}
Note that here $\varphi_{c}(z)$ is the inverse of the B\"ottcher
coordinate, not the B\"ottcher coordinate itself. Let $x=z^{-d}$,
then (\ref{eq1.2}) can be simplified as
\begin{equation}
\left(1+\sum_{n=1}^{\infty}a_{n}x^{n}\right)^{d}=1+cx+\sum_{n=1}^{\infty}a_{n}x^{nd}.\label{eq1.3}
\end{equation}
Let $g_{c}(x)=x^{d}+cx^{d+1}$ for some $c\in\mathbb{C}_{p}$ and
let
\[
\phi_{c}(x)=x\left(1+\sum_{n=1}^{\infty}b_{n}x^{n}\right)
\]
satisfy the B\"ottcher equation
\[
\phi_{c}(g_{c}(x))=\phi_{c}(x)^{d}.
\]
Then it can be simplified as
\begin{equation}
\left(1+\sum_{n=1}^{\infty}b_{n}x^{n}\right)^{d}=1+cx+\sum_{n=1}^{\infty}b_{n}x^{nd}(1+cx)^{n+1}.\label{eq1.4}
\end{equation}

Instead of working on $g(x)$ and $\phi(x)$ directly, we will work
on their generalizations $g_{c}(x)$ and $\phi_{c}(x)$. Therefore,
we need to study (\ref{eq1.4}) and, in particular, the properties
of $v_{p}(b_{n})$, where $v_{p}$ is the $p$-adic valuation in $\mathbb{C}_{p}$.
The key idea of our proofs is to consider (\ref{eq1.4}) as a \textbf{perturbation}
of (\ref{eq1.3}). First we show that under some conditions on $d$
and $c$, the values of $v_{p}(a_{n})$ can be explicitly obtained.
Then we show that under the same conditions, the perturbation is small
enough so that $v_{p}(b_{n})=v_{p}(a_{n})$, which enables us to determine
the radii of convergence of $\phi_{c}(x)$ and $\phi_{c}^{-1}(x)$.

The conditions mentioned above can be summarized as follows:

\begin{ConditionAB} \label{condA} Assume that $p$ is a prime number,
$N=0$, $d$ is a \textbf{multiple} of $p$, and
\begin{equation}
v_{p}(c)<v_{p}(d)+\frac{v_{p}((d-1)!)}{d-1}.\label{eq1.5}
\end{equation}

\end{ConditionAB}

\begin{ConditionAB} \label{condB} Assume that $p$ is a prime number,
$N\geq1$ is an integer, $d$ is a \textbf{power} of $p$, and
\begin{equation}
Nv_{p}(d)+\frac{v_{p}((d-1)!)}{d-1}<v_{p}(c)<(N+1)v_{p}(d)+\frac{v_{p}((d-1)!)}{d-1}.\label{eq1.6}
\end{equation}

\end{ConditionAB}

Now we are ready to give the main theorems of this article.
\begin{thm}
\label{thm1.2} Let $p$, $N$, $d$, and $c$ satisfy Condition \ref{condA}
or \ref{condB}. Then the maximal convergent open disks of $\varphi_{c}(z)$
and $\varphi_{c}^{-1}(z)$ are both $D(\infty,r_{N}^{1/d})=\{z\in\mathbb{C}_{p}:|z|_{p}>r_{N}^{1/d}\}$,
where
\[
r_{N}=\left(|c/d^{N+1}|_{p}p^{1/(p-1)}\right)^{1/d^{N}}>1.
\]
Moreover, $\varphi_{c}(z)$ gives a bijective isometry from $D(\infty,r_{N}^{1/d})$
onto itself.
\end{thm}

\begin{thm}
\label{thm1.3} Let $p$, $N$, $d$, and $c$ satisfy Condition \ref{condA}
or \ref{condB}. Then the maximal convergent open disks of $\phi_{c}(x)$
and $\phi_{c}^{-1}(x)$ are both $D(0,r_{N}^{-1})=\{x\in\mathbb{C}_{p}:|x|_{p}<r_{N}^{-1}\}$.
Moreover, $\phi_{c}(x)$ gives a bijective isometry from $D(0,r_{N}^{-1})$
onto itself.
\end{thm}

In Conjecture \ref{conj1.1}, we have $d=p^{2}$ and $c=p^{r+2}$
with $r\in\mathbb{Z}_{\geq0}$, so we can take $N=\left\lfloor (r+1)/2\right\rfloor $
to satisfy Condition \ref{condA} or \ref{condB}. Then by Theorem
\ref{thm1.3}, the radius of convergence of $\phi^{-1}(x)$ is
\[
r_{N}^{-1}=\left(|c/d^{N+1}|_{p}p^{1/(p-1)}\right)^{-1/d^{N}}=p^{-p^{-r}/(p-1)},
\]
as conjectured by Salerno and Silverman.
\begin{cor}
\label{cor1.4} Conjecture \ref{conj1.1} is true.
\end{cor}

We remark that the technical Conditions \ref{condA} and \ref{condB}
are crucial for Theorems \ref{thm1.2} and \ref{thm1.3}.
\begin{rem}
\label{rem1.5} If $p=d=c=2$, then $f_{2}(z)=z^{2}-2$ is a Chebyshev
map. Now $\varphi_{2}(z)=z+z^{-1}$ and
\[
\varphi_{2}^{-1}(z)=z\left(1-\sum_{n=1}^{\infty}\frac{C_{n}}{z^{2n}}\right),\text{ where }C_{n}=\frac{(2n-2)!}{(n-1)!n!}
\]
are known as the Catalan numbers. Their maximal convergent open disks
are $D(\infty,0)$ and $D(\infty,1)$, respectively. On the other
hand, the maximal convergent open disks of $\varphi_{c}(z)$ and $\varphi_{c}^{-1}(z)$
in Theorem \ref{thm1.2} are always identical.
\end{rem}

\begin{rem}
\label{rem1.6} If $d=p$ and $c$ is a multiple of $p$, then by
\cite[Theorem 4]{SS20}, both $\phi_{c}(x)$ and $\phi_{c}^{-1}(x)$
have integral coefficients so that they are convergent on the open
unit disk $D(0,1)$. On the other hand, $D(0,r_{N}^{-1})$ in Theorem
\ref{thm1.3} is always strictly smaller than $D(0,1)$.
\end{rem}

Theorem \ref{thm1.2} can also be interpreted in a different way.
For any $c\in\mathbb{C}_{p}$, let
\[
B(c)=\{z\in\mathbb{C}_{p}:f_{c}^{\circ n}(z)\to\infty\text{ as }n\to\infty\}
\]
be the \textbf{basin of infinity} of $f_{c}(z)$. We say that $B(c_{1})$
and $B(c_{2})$ are \textbf{analytically conjugate} if there is a
bijective analytic map $\Phi_{c_{1},c_{2}}:B(c_{1})\to B(c_{2})$
whose inverse is also analytic such that
\begin{equation}
f_{c_{2}}(\Phi_{c_{1},c_{2}}(z))=\Phi_{c_{1},c_{2}}(f_{c_{1}}(z)).\label{eq1.7}
\end{equation}
We know that $f_{c}(z)$ has good reduction if and only if
\[
v_{p}(c)\geq0\Leftrightarrow B(c)=D(\infty,1)\Leftrightarrow0\notin B(c),
\]
so Theorem \ref{thm1.2} tells us $\varphi_{c}(z)$ does not give
an analytic conjugacy between $B(0)$ and $B(c)$. Indeed, we are
able to prove the following more general result.
\begin{thm}
\label{thm1.7} Let $p$, $N$, $d$, and $c=c_{2}$ satisfy Condition
\ref{condA} or \ref{condB}. Let $c_{1}$ satisfy $v_{p}(c_{1})\geq0$
and
\begin{equation}
v_{p}(c_{1}^{d-1}-c_{2}^{d-1})=v_{p}(c_{2}^{d-1}).\label{eq1.8}
\end{equation}
Then $B(c_{1})$ and $B(c_{2})$ are not analytically conjugate.
\end{thm}

We remark that Theorem \ref{thm1.7} is inspired by the work of DeMarco
and Pilgrim \cite{DP11}, although in this article we only consider
the most basic cases. A discussion for the analytic conjugacy between
the basins of infinity of two tame polynomials can be found in \cite{KN22}.

The structure of this article is as follows: In Section \ref{sec2},
we prove some preliminary lemmas which will be needed later. In Sections
\ref{sec3}, \ref{sec4}, and \ref{sec5}, we prove Theorems \ref{thm1.2},
\ref{thm1.3}, and \ref{thm1.7}, respectively.

\textbf{Acknowledgments.} The second named author would like to thank
Prof. Julie Tzu-Yueh Wang and Institute of Mathematics, Academia Sinica
for their hospitality during his visit in 2021. We would also like
to thank Prof. Joseph H. Silverman for his comments on an earlier
draft of this article.

\section{\label{sec2} Some preliminary lemmas}

In this section, we prove some preliminary lemmas which will be needed
later.
\begin{lem}
\label{lem2.1} We have $(d-1)!^{kn_{k}}(dk!)^{n_{k}}n_{k}!$ divides
$(dkn_{k})!$ for any $d,k\geq1$ and $n_{k}\geq0$.
\end{lem}

\begin{proof}
We have
\[
(d-1)!^{kn_{k}}(dk!)^{n_{k}}n_{k}!=\prod_{i=1}^{n_{k}}(idk)(d-1)!^{k}(k-1)!
\]
divides
\[
\prod_{i=1}^{n_{k}}\left((idk)\prod_{j=(i-1)dk+1}^{idk-1}j\right)=(dkn_{k})!.\qedhere
\]
\end{proof}
\begin{lem}
[Legendre] \label{lem2.2} Let $s_{p}(n)$ be the sum of the digits
in the base-$p$ expansion of $n$. Then
\[
v_{p}(n!)=\frac{n-s_{p}(n)}{p-1}.
\]
\end{lem}

\begin{lem}
\label{lem2.3} Let $p$ be a prime number and let $d$ be a power
of $p$. If $n_{0}+n_{1}d=m_{0}+m_{1}d$ for some $0\leq n_{0}<d$
and $n_{1},m_{0},m_{1}\geq0$, then
\[
v_{p}\left(\frac{m_{0}!m_{1}!}{n_{0}!n_{1}!}\right)\leq(n_{1}-m_{1})v_{p}(d!).
\]
\end{lem}

\begin{proof}
By Lemma \ref{lem2.2}, we have
\[
\text{LHS}=\frac{(m_{0}-n_{0}+m_{1}-n_{1})-(s_{p}(m_{0})-s_{p}(n_{0})+s_{p}(m_{1})-s_{p}(n_{1}))}{p-1}
\]
and
\[
\text{RHS}=\frac{(n_{1}-m_{1})(d-1)}{p-1}=\frac{m_{0}-n_{0}+m_{1}-n_{1}}{p-1}.
\]
Since $d$ is a power of $p$, the base-$p$ and base-$d$ expansions
are compatible. Hence,
\begin{align*}
(p-1)(\text{RHS}-\text{LHS}) & =s_{p}(m_{0})-s_{p}(n_{0})+s_{p}(m_{1})-s_{p}(n_{1})\\
 & =s_{p}(m_{0}-n_{0})+s_{p}(m_{1})-s_{p}(n_{1})\\
 & =s_{p}((n_{1}-m_{1})d)+s_{p}(m_{1})-s_{p}(n_{1})\\
 & =s_{p}(n_{1}-m_{1})+s_{p}(m_{1})-s_{p}(n_{1})\geq0.\qedhere
\end{align*}
\end{proof}
\begin{lem}
\label{lem2.4} Let $p$ be a prime number, let $d$ be a power of
$p$, and let $N$ be a non-negative integer. If $n\geq1$ can be
decomposed as
\[
n=\sum_{k=0}^{N}n_{k}d^{k}\text{ with }0\leq n_{k}<d\text{ for any }0\leq k<N\text{ and }n_{N}\geq0,
\]
then
\[
v_{p}(n!)=\sum_{k=0}^{N}v_{p}(d^{k}!^{n_{k}}n_{k}!).
\]
\end{lem}

\begin{proof}
Since $d$ is a power of $p$, the base-$p$ and base-$d$ expansions
are compatible. Hence, by Lemma \ref{lem2.2}, we have
\[
v_{p}(n!)=\frac{n-s_{p}(n)}{p-1}=\sum_{k=0}^{N}\frac{n_{k}d^{k}-s_{p}(n_{k}d^{k})}{p-1}
\]
and
\[
\sum_{k=0}^{N}v_{p}(d^{k}!^{n_{k}}n_{k}!)=\sum_{k=0}^{N}\frac{n_{k}(d^{k}-1)+n_{k}-s_{p}(n_{k})}{p-1}=\sum_{k=0}^{N}\frac{n_{k}d^{k}-s_{p}(n_{k}d^{k})}{p-1},
\]
which are equal.
\end{proof}
\begin{lem}
\label{lem2.5} Let $d\in\mathbb{Z}\backslash\{0\}$ and let
\[
F(z)=z\left(1+\sum_{n=1}^{\infty}\frac{\alpha_{n}}{z^{nd}}\right)
\]
be a formal power series. Then
\[
F^{-1}(z)=z\left(1+\sum_{n=1}^{\infty}\frac{\beta_{n}}{z^{nd}}\right),
\]
where
\[
\beta_{n}=-\frac{1}{nd-1}\sum_{\sum_{k=1}^{n}km_{k}=n}\left({nd-1 \choose \sum_{k=1}^{n}m_{k}}{\sum_{k=1}^{n}m_{k} \choose m_{1},\dots,m_{n}}\prod_{k=1}^{n}\alpha_{k}^{m_{k}}\right).
\]
\end{lem}

\begin{proof}
Let $[z^{n}]F^{-1}(z)$ be the coefficient of $z^{n}$ in $F^{-1}(z)$.
By the Lagrange\textendash B\"{u}rmann Formula,
\begin{align*}
\beta_{n}=[z^{-nd+1}]F^{-1}(z) & =\frac{1}{-nd+1}[z^{-nd}]\left(\frac{z}{F(z)}\right)^{-nd+1}\\
 & =-\frac{1}{nd-1}[z^{-nd}]\left(1+\sum_{k=1}^{\infty}\frac{\alpha_{k}}{z^{kd}}\right)^{nd-1}.
\end{align*}
Then we expand this power series to get the result.
\end{proof}

\section{\label{sec3} Proof of Theorem \ref{thm1.2}}

In this section, we focus on the properties of $a_{n}$ and give the
proof of Theorem \ref{thm1.2}. First we show that we can compute
$a_{n}$ inductively from (\ref{eq1.3}).
\begin{prop}
\label{prop3.1} The sequence $a_{n}$ satisfies the following inductive
relations:
\begin{enumerate}
\item For any $1\leq n<d$, we have $a_{n}={1/d \choose n}c^{n}$.
\item For any $d^{i}\leq n<d^{i+1}$ with $i\geq1$, we have
\[
a_{n}=\sum_{n_{0}+d\sum_{k=1}^{d^{i}-1}kn_{k}=n}\alpha(n_{0},n_{1},\dots,n_{d^{i}-1}),
\]
where the summation is taken over all non-negative $d^{i}$-tuples
$(n_{0},n_{1},\dots,n_{d^{i}-1})$ such that
\begin{equation}
n_{0}+d\sum_{k=1}^{d^{i}-1}kn_{k}=n\label{eq3.1}
\end{equation}
and
\[
\alpha(n_{0},n_{1},\dots,n_{d^{i}-1})=\frac{c^{n_{0}}}{d^{n_{0}}n_{0}!}\prod_{k=1}^{d^{i}-1}\frac{a_{k}^{n_{k}}}{d^{n_{k}}n_{k}!}\prod_{j=0}^{\sum_{k=0}^{d^{i}-1}n_{k}-1}(1-jd).
\]
\end{enumerate}
\end{prop}

\begin{proof}
Let
\begin{equation}
\left(1+\sum_{n=1}^{\infty}a_{n}'x^{n}\right)^{d}=1+cx,\label{eq3.2}
\end{equation}
then
\[
1+\sum_{n=1}^{\infty}a_{n}'x^{n}=(1+cx)^{1/d}=1+\sum_{n=1}^{\infty}{1/d \choose n}c^{n}x^{n}
\]
and $a_{n}'={1/d \choose n}c^{n}\text{ for any }n\geq1$. Considering
the difference of (\ref{eq1.3}) and (\ref{eq3.2}), we get
\[
\left(\sum_{n=1}^{\infty}(a_{n}-a_{n}')x^{n}\right)\left(\sum_{i=0}^{d-1}\left(1+\sum_{n=1}^{\infty}a_{n}x^{n}\right)^{i}\left(1+\sum_{n=1}^{\infty}a_{n}'x^{n}\right)^{d-1-i}\right)=\sum_{n=1}^{\infty}a_{n}x^{nd}.
\]
Comparing the degrees on both sides, we get $a_{n}=a_{n}'={1/d \choose n}c^{n}$
for any $1\leq n<d$. Moreover, let
\[
\left(1+\sum_{n=1}^{\infty}a_{n}''x^{n}\right)^{d}=1+cx+\sum_{n=1}^{d^{i}-1}a_{n}x^{nd},
\]
then
\[
1+\sum_{n=1}^{\infty}a_{n}''x^{n}=1+\sum_{j=1}^{\infty}{1/d \choose j}\left(cx+\sum_{n=1}^{d^{i}-1}a_{n}x^{nd}\right)^{j}.
\]
By the same reasoning as above, for any $d^{i}\leq n<d^{i+1}$, we
have
\begin{align*}
a_{n}=a_{n}'' & =\sum_{n_{0}+d\sum_{k=1}^{d^{i}-1}kn_{k}=n}{1/d \choose \sum_{k=0}^{d^{i}-1}n_{k}}{\sum_{k=0}^{d^{i}-1}n_{k} \choose n_{0},n_{1},\dots,n_{d^{i}-1}}c^{n_{0}}\prod_{k=1}^{d^{i}-1}a_{k}^{n_{k}}\\
 & =\sum_{n_{0}+d\sum_{k=1}^{d^{i}-1}kn_{k}=n}\left(\frac{c^{n_{0}}}{d^{n_{0}}n_{0}!}\prod_{k=1}^{d^{i}-1}\frac{a_{k}^{n_{k}}}{d^{n_{k}}n_{k}!}\prod_{j=0}^{\sum_{k=0}^{d^{i}-1}n_{k}-1}(1-jd)\right).\qedhere
\end{align*}

An immediate corollary of Proposition \ref{prop3.1} is that $a_{n}$
can be considered as a polynomial of degree $n$ in $c$. This corollary,
however, will not be used in the sequel. More results of this type
can be found in \cite[Section 2.4.1]{FG22}.
\end{proof}
\begin{cor}
\label{cor3.2} For any $n\geq1$, we have $a_{n}\in\frac{1}{n!}\mathbb{Z}[c/d]$
with the leading term ${1/d \choose n}c^{n}$.
\end{cor}

\begin{proof}
By Proposition \ref{prop3.1}, the assertion is true for any $1\leq n<d$.
Now we assume that it is true for any $1\leq n<d^{i}$ and use induction
to show that it is also true for any $d^{i}\leq n<d^{i+1}$. For each
$(n_{0},n_{1},\dots,n_{d^{i}-1})$ such that (\ref{eq3.1}) holds
and $n_{0}\neq n$, we have
\[
\text{deg}_{c}\alpha(n_{0},n_{1},\dots,n_{d^{i}-1})=n_{0}+\sum_{k=1}^{d^{i}-1}kn_{k}<n.
\]
Hence, the leading term of $a_{n}$ is given by $\alpha(n,0,\dots,0)={1/d \choose n}c^{n}$.
Also, by the induction hypothesis, we know that
\begin{align*}
\alpha(n_{0},n_{1},\dots,n_{d^{i}-1}) & \in\frac{1}{n_{0}!}\prod_{k=1}^{d^{i}-1}\frac{1}{(dk!)^{n_{k}}n_{k}!}\mathbb{Z}[c/d]\\
 & \subseteq\frac{1}{n_{0}!}\prod_{k=1}^{d^{i}-1}\frac{1}{(dkn_{k})!}\mathbb{Z}[c/d]\text{ by Lemma \ref{lem2.1},}\\
 & \subseteq\frac{1}{n!}\mathbb{Z}[c/d]\text{ by \eqref{eq3.1}.}
\end{align*}
This completes the proof.
\end{proof}
The following proposition is the most important step of this article.
It shows that under Condition \ref{condA} or \ref{condB}, we are
able to obtain all values of $v_{p}(a_{n})$ simultaneously rathen
than successively.
\begin{prop}
\label{prop3.3} Let $p$, $N$, $d$, and $c$ satisfy Condition
\ref{condA} or \ref{condB}. Then
\begin{enumerate}
\item For any $0\leq k\leq N$, we have
\[
v_{p}(a_{d^{k}})=v_{p}\left(\frac{c}{d^{k+1}}\right).
\]
\item If $n\geq1$ can be decomposed as
\begin{equation}
n=\sum_{k=0}^{N}n_{k}d^{k}\text{ with }0\leq n_{k}<d\text{ for any }0\leq k<N\text{ and }n_{N}\geq0,\label{eq3.3}
\end{equation}
then
\[
v_{p}(a_{n})=\sum_{k=0}^{N}v_{p}\left(\frac{a_{d^{k}}^{n_{k}}}{n_{k}!}\right).
\]
\item Consequently, for any $n\geq1$, we have
\[
v_{p}(a_{n})=v_{p}\left(\frac{c^{n}}{d^{n}n!}\right)-\sum_{k=1}^{N}\left((d-1)v_{p}\left(\frac{c}{d^{k}}\right)-v_{p}((d-1)!)\right)\left\lfloor \frac{n}{d^{k}}\right\rfloor .
\]
\end{enumerate}
\end{prop}

\begin{proof}
By Proposition \ref{prop3.1}, the assertions are true for any $1\leq n<d$.
Now we assume that they are true for any $1\leq n<d^{i}$ and use
induction to show that they are also true for any $d^{i}\leq n<d^{i+1}$.

We know that each partition $\sigma$ of $n$ with a particular form
gives a summand $\alpha(\sigma)$ of $a_{n}$. We call (\ref{eq3.3})
the canonical partition $\sigma_{\text{can}}$ of $n$. We claim that
$v_{p}(\alpha(\sigma))>v_{p}(\alpha(\sigma_{\text{can}}))$ unless
$\sigma=\sigma_{\text{can}}$.

Let $\sigma$ be an arbitrary partition $n=m_{0}+d\sum_{j=1}^{d^{i}-1}jm_{j}$
and, for each $j$, let $j=\sum_{k=0}^{N}m_{j,k}d^{k}$ be the canonical
partition of $j$. Then we can produce another partition $\sigma_{0}$
which is given by
\begin{align*}
n & =m_{0}+d\sum_{j=1}^{d^{i}-1}jm_{j}=m_{0}+d\sum_{j=1}^{d^{i}-1}\left(\sum_{k=0}^{N}m_{j,k}d^{k}\right)m_{j}\\
 & =m_{0}+d\sum_{k=0}^{N}\left(d^{k}\sum_{j=1}^{d^{i}-1}m_{j}m_{j,k}\right)=m_{0}+d\sum_{k=0}^{N}d^{k}M_{d^{k}},
\end{align*}
where
\begin{equation}
M_{d^{k}}=\sum_{j=1}^{d^{i}-1}m_{j}m_{j,k}.\label{eq3.4}
\end{equation}
Now
\begin{align*}
v_{p} & (\alpha(\sigma))=v_{p}\left(\frac{c^{m_{0}}}{d^{m_{0}}m_{0}!}\prod_{j=1}^{d^{i}-1}\frac{a_{j}^{m_{j}}}{d^{m_{j}}m_{j}!}\right)\text{ since }p\mid d,\\
 & =v_{p}\left(\frac{c^{m_{0}}}{d^{m_{0}}m_{0}!}\right)+\sum_{j=1}^{d^{i}-1}\left(m_{j}\sum_{k=0}^{N}v_{p}\left(\frac{a_{d^{k}}^{m_{j,k}}}{m_{j,k}!}\right)-v_{p}(d^{m_{j}}m_{j}!)\right)\text{ by induction,}\\
 & =v_{p}\left(\frac{c^{m_{0}}}{d^{m_{0}}m_{0}!}\right)+\sum_{k=0}^{N}v_{p}(a_{d^{k}}^{M_{d^{k}}})-\sum_{k=0}^{N}\sum_{j=1}^{d^{i}-1}v_{p}(m_{j,k}!^{m_{j}})-\sum_{j=1}^{d^{i}-1}v_{p}(d^{m_{j}}m_{j}!)
\end{align*}
and
\[
v_{p}(\alpha(\sigma_{0}))=v_{p}\left(\frac{c^{m_{0}}}{d^{m_{0}}m_{0}!}\right)+\sum_{k=0}^{N}v_{p}(a_{d^{k}}^{M_{d^{k}}})-\sum_{k=0}^{N}v_{p}(d^{M_{d^{k}}}M_{d^{k}}!).
\]
If $\sigma\neq\sigma_{0}$, then there is some $j\notin\{d^{k}:0\leq k\leq N\}$
such that $m_{j}\neq0$. Therefore,
\begin{align*}
v_{p} & (\alpha(\sigma))-v_{p}(\alpha(\sigma_{0}))=\sum_{k=0}^{N}v_{p}(d^{M_{d^{k}}}M_{d^{k}}!)-\sum_{k=0}^{N}\sum_{j=1}^{d^{i}-1}v_{p}(m_{j,k}!^{m_{j}})-\sum_{j=1}^{d^{i}-1}v_{p}(d^{m_{j}}m_{j}!)\\
 & \geq\sum_{k=0}^{N}v_{p}(d^{M_{d^{k}}}M_{d^{k}}!)-\sum_{k=0}^{N}\sum_{j=1}^{d^{i}-1}v_{p}(m_{j,k}!^{m_{j}})-\sum_{j=1}^{d^{i}-1}v_{p}(d^{m_{j}})-\sum_{k=0}^{N}\sum_{\substack{j=1\\
m_{j,k}\neq0
}
}^{d^{i}-1}v_{p}(m_{j}!)\\
 & =\sum_{j=1}^{d^{i}-1}\left(\sum_{k=0}^{N}m_{j,k}-1\right)m_{j}v_{p}(d)+\sum_{k=0}^{N}\left(v_{p}(M_{d^{k}}!)-\sum_{\substack{j=1\\
m_{j,k}\neq0
}
}^{d^{i}-1}v_{p}(m_{j,k}!^{m_{j}}m_{j}!)\right)\\
 & \geq\sum_{j=1}^{d^{i}-1}\left(\sum_{k=0}^{N}m_{j,k}-1\right)m_{j}v_{p}(d)\text{ by Lemma \ref{lem2.1} and \eqref{eq3.4},}\\
 & >0\text{ since }\sigma\neq\sigma_{0}.
\end{align*}
Next, for each $1\leq j\leq N$, we let $\sigma_{j}$ be the partition
\[
n=\sum_{k=0}^{j-1}n_{k}d^{k}+N_{j}d^{j}+\sum_{k=j}^{N}M_{d^{k}}d^{k+1}.
\]
We also let $N_{0}=m_{0}$ and $a_{d^{-1}}=c$. For any $1\leq j\leq N$,
if $\sigma_{j-1}\neq\sigma_{j}$, then we have
\begin{equation}
N_{j-1}+M_{d^{j-1}}d=n_{j-1}+N_{j}d\label{eq3.5}
\end{equation}
and
\begin{align*}
v_{p} & (\alpha(\sigma_{j-1}))-v_{p}(\alpha(\sigma_{j}))=v_{p}\left(\frac{a_{d^{j-2}}^{N_{j-1}}}{d^{N_{j-1}}N_{j-1}!}\frac{a_{d^{j-1}}^{M_{d^{j-1}}}}{d^{M_{d^{j-1}}}M_{d^{j-1}}!}\right)-v_{p}\left(\frac{a_{d^{j-2}}^{n_{j-1}}}{d^{n_{j-1}}n_{j-1}!}\frac{a_{d^{j-1}}^{N_{j}}}{d^{N_{j}}N_{j}!}\right)\\
 & =v_{p}\left(\frac{(c/d^{j})^{N_{j-1}}}{N_{j-1}!}\frac{(c/d^{j})^{M_{d^{j-1}}}}{d^{M_{d^{j-1}}}M_{d^{j-1}}!}\right)-v_{p}\left(\frac{(c/d^{j})^{n_{j-1}}}{n_{j-1}!}\frac{(c/d^{j})^{N_{j}}}{d^{N_{j}}N_{j}!}\right)\text{ by induction,}\\
 & =(N_{j}-M_{d^{j-1}})\left((d-1)v_{p}\left(\frac{c}{d^{j}}\right)+v_{p}(d)\right)-v_{p}\left(\frac{N_{j-1}!M_{d^{j-1}}!}{n_{j-1}!N_{j}!}\right)\text{ by \eqref{eq3.5},}\\
 & >(N_{j}-M_{d^{j-1}})v_{p}(d!)-v_{p}\left(\frac{N_{j-1}!M_{d^{j-1}}!}{n_{j-1}!N_{j}!}\right)\text{ by the LHS of \eqref{eq1.6} and }\sigma_{j-1}\neq\sigma_{j},\\
 & \geq0\text{ by Lemma \ref{lem2.3}. (Here we need the condition }d\text{ is a power of }p\text{.)}
\end{align*}
Next, if $\sigma_{N}\neq\sigma_{\text{can}}$, then by the same reasoning
as above, we have
\begin{align*}
v_{p} & (\alpha(\sigma_{N}))-v_{p}(\alpha(\sigma_{\text{can}}))=-M_{d^{N}}\left((d-1)v_{p}\left(\frac{c}{d^{N+1}}\right)+v_{p}(d)\right)-v_{p}\left(\frac{N_{N}!M_{d^{N}}!}{n_{N}!}\right)\\
 & >v_{p}\left(\frac{n_{N}!}{N_{N}!M_{d^{N}}!d!^{M_{d^{N}}}}\right)\text{ by \eqref{eq1.5}, the RHS of \eqref{eq1.6}, and }\sigma_{N}\neq\sigma_{\text{can}},\\
 & \geq0\text{ by Lemma \ref{lem2.1}.}
\end{align*}
We have shown that $v_{p}(\alpha(\sigma))>v_{p}(\alpha(\sigma_{0}))>\dots>v_{p}(\alpha(\sigma_{N})>v_{p}(\alpha(\sigma_{\text{can}}))$,
so
\[
v_{p}(a_{n})=v_{p}(\alpha(\sigma_{\text{can}}))=v_{p}\left(\frac{c^{n_{0}}}{d^{n_{0}}n_{0}!}\right)+\sum_{k=1}^{N}v_{p}\left(\frac{a_{d^{k-1}}^{n_{k}}}{d^{n_{k}}n_{k}!}\right),
\]
which implies parts (1) and (2) immediately. Part (3) is a corollary
of parts (1) and (2) because
\begin{align*}
v_{p}(a_{n}) & =\sum_{k=0}^{N}v_{p}\left(\frac{a_{d^{k}}^{n_{k}}}{n_{k}!}\right)=\sum_{k=0}^{N}n_{k}v_{p}\left(\frac{c}{d^{k+1}}\right)-\sum_{k=0}^{N}v_{p}(n_{k}!)\\
 & =\sum_{k=0}^{N}n_{k}v_{p}\left(\frac{c}{d^{k+1}}\right)+\sum_{k=0}^{N}n_{k}v_{p}(d^{k}!)-v_{p}(n!)\text{ by Lemma \ref{lem2.4},}\\
 & =\sum_{k=0}^{N-1}\left(\left\lfloor \frac{n}{d^{k}}\right\rfloor -\left\lfloor \frac{n}{d^{k+1}}\right\rfloor d\right)v_{p}\left(\frac{cd^{k}!}{d^{k+1}}\right)+\left\lfloor \frac{n}{d^{N}}\right\rfloor v_{p}\left(\frac{cd^{N}!}{d^{N+1}}\right)-v_{p}(n!)\\
 & =v_{p}\left(\frac{c^{n}}{d^{n}n!}\right)+\sum_{k=1}^{N}\left\lfloor \frac{n}{d^{k}}\right\rfloor v_{p}\left(\frac{cd^{k}!}{d^{k+1}}\right)-\sum_{k=1}^{N}\left\lfloor \frac{n}{d^{k}}\right\rfloor dv_{p}\left(\frac{cd^{k-1}!}{d^{k}}\right)\\
 & =v_{p}\left(\frac{c^{n}}{d^{n}n!}\right)-\sum_{k=1}^{N}\left\lfloor \frac{n}{d^{k}}\right\rfloor \left((d-1)v_{p}\left(\frac{c}{d^{k}}\right)-v_{p}\left(\frac{d^{k}!}{d(d^{k-1}!)^{d}}\right)\right)\\
 & =v_{p}\left(\frac{c^{n}}{d^{n}n!}\right)-\sum_{k=1}^{N}\left\lfloor \frac{n}{d^{k}}\right\rfloor \left((d-1)v_{p}\left(\frac{c}{d^{k}}\right)-v_{p}((d-1)!)\right).
\end{align*}
This completes the proof.
\end{proof}
From Proposition \ref{prop3.3}, we can deduce that the sequence $v_{p}(a_{n})/n$
has a negative limit.
\begin{prop}
\label{prop3.4} Let $p$, $N$, $d$, and $c$ satisfy Condition
\ref{condA} or \ref{condB}. Then the sequence $v_{p}(a_{n})$ is
subadditive and
\[
\lim_{n\to\infty}\frac{v_{p}(a_{n})}{n}=\inf_{n}\frac{v_{p}(a_{n})}{n}=\frac{v_{p}(c/d^{N+1})}{d^{N}}-\frac{1}{(p-1)d^{N}}<0.
\]
\end{prop}

\begin{proof}
The subadditivity of $v_{p}(a_{n})$ can be easily seen from Proposition
\ref{prop3.3}. Therefore, by Fekete's Lemma, the limit of $v_{p}(a_{n})/n$
exists and is equal to the infimum of $v_{p}(a_{n})/n$. By Proposition
\ref{prop3.3} and Lemma \ref{lem2.2},
\begin{align*}
\inf_{n}\frac{v_{p}(a_{n})}{n} & =\inf_{n}\left(v_{p}\left(\frac{c}{d}\right)-\frac{n-s_{p}(n)}{(p-1)n}-\frac{1}{n}\sum_{k=1}^{N}\left((d-1)v_{p}\left(\frac{c}{d^{k}}\right)-v_{p}((d-1)!)\right)\left\lfloor \frac{n}{d^{k}}\right\rfloor \right)\\
 & =v_{p}\left(\frac{c}{d}\right)-\frac{1}{p-1}-\sum_{k=1}^{N}\left((d-1)v_{p}\left(\frac{c}{d^{k}}\right)-v_{p}((d-1)!)\right)\frac{1}{d^{k}}\\
 & =\frac{v_{p}(a_{d^{N}})}{d^{N}}-\frac{1}{(p-1)d^{N}}=\frac{v_{p}(c/d^{N+1})}{d^{N}}-\frac{1}{(p-1)d^{N}}.
\end{align*}
Moreover, the limit is negative because
\begin{align*}
\frac{v_{p}(c/d^{N+1})}{d^{N}} & <\frac{v_{p}((d-1)!)}{(d-1)d^{N}}\text{ by \eqref{eq1.5} and the RHS of \eqref{eq1.6},}\\
 & =\frac{(d-1)-s_{p}(d-1)}{(p-1)(d-1)d^{N}}<\frac{1}{(p-1)d^{N}}\text{ by Lemma \ref{lem2.2}.}
\end{align*}
This completes the proof.
\end{proof}
The last ingredient needed for the proof of Theorem \ref{thm1.2}
is the following inequality.
\begin{prop}
\label{prop3.5} Let $p$, $N$, $d$, and $c$ satisfy Condition
\ref{condA} or \ref{condB}. If $n=\sum_{k=1}^{n}km_{k}$, where
$m_{k}\geq0$ for any $1\leq k\leq n$, then
\[
v_{p}(a_{n})\leq\sum_{k=1}^{n}v_{p}\left(\frac{a_{k}^{m_{k}}}{m_{k}!}\right).
\]
\end{prop}

\begin{proof}
Let
\begin{equation}
e(n)=\sum_{k=1}^{N}\left((d-1)v_{p}\left(\frac{c}{d^{k}}\right)-v_{p}((d-1)!)\right)\left\lfloor \frac{n}{d^{k}}\right\rfloor .\label{eq3.6}
\end{equation}
It is clear that the sequence $e(n)$ is superadditive. Then
\begin{align*}
\sum_{k=1}^{n}v_{p}\left(\frac{a_{k}^{m_{k}}}{m_{k}!}\right) & =\sum_{k=1}^{n}v_{p}\left(\frac{c^{km_{k}}}{d^{km_{k}}}\right)-\sum_{k=1}^{n}v_{p}(k!^{m_{k}}m_{k}!)-\sum_{k=1}^{n}m_{k}e(k)\text{ by Proposition \ref{prop3.3},}\\
 & \geq v_{p}\left(\frac{c^{n}}{d^{n}}\right)-v_{p}(n!)-e(n)\text{ by Lemma \ref{lem2.1},}\\
 & =v_{p}(a_{n})\text{ by Proposition \ref{prop3.3}.}\qedhere
\end{align*}
\end{proof}
Now we are ready to give the proof of Theorem \ref{thm1.2}.
\begin{proof}
[Proof of Theorem \ref{thm1.2}] By (\ref{eq1.1}) and Proposition
\ref{prop3.4}, $\varphi_{c}(z)$ is convergent when
\[
|z|_{p}^{d}>\lim_{n\to\infty}|a_{n}|_{p}^{1/n}=\lim_{n\to\infty}p^{-v_{p}(a_{n})/n}=r_{N}.
\]
By Lemma \ref{lem2.5},
\[
\varphi_{c}^{-1}(z)=z\left(1+\sum_{n=1}^{\infty}\frac{a_{n}'}{z^{nd}}\right),
\]
where
\[
a_{n}'=-\sum_{\sum_{k=1}^{n}km_{k}=n}\left(\prod_{j=2}^{\sum_{k=1}^{n}m_{k}}(nd-j)\prod_{k=1}^{n}\frac{a_{k}^{m_{k}}}{m_{k}!}\right).
\]
By Proposition \ref{prop3.5}, $v_{p}(a_{n}')\geq v_{p}(a_{n})$ for
any $n\geq1$. Now we want to show that $v_{p}(a_{n}')=v_{p}(a_{n})$
for infinitely many $n$, which will then imply
\[
\liminf_{n\to\infty}\frac{v_{p}(a_{n}')}{n}=\liminf_{n\to\infty}\frac{v_{p}(a_{n})}{n}
\]
and the maximal convergent open disks of $\varphi_{c}(z)$ and $\varphi_{c}^{-1}(z)$
are the same. We claim that if $n$ is a power of $p$ and $m_{n}=0$,
then
\[
v_{p}\left(\prod_{j=2}^{\sum_{k=1}^{n}m_{k}}(nd-j)\prod_{k=1}^{n}\frac{a_{k}^{m_{k}}}{m_{k}!}\right)>v_{p}(a_{n}).
\]
Suppose not, then by Propositions \ref{prop3.3} and \ref{prop3.5},
\begin{align*}
0 & =v_{p}\left(\prod_{j=2}^{\sum_{k=1}^{n}m_{k}}(nd-j)\prod_{k=1}^{n}\frac{a_{k}^{m_{k}}}{m_{k}!}\right)-v_{p}(a_{n})\\
 & =\sum_{j=2}^{\sum_{k=1}^{n}m_{k}}v_{p}(nd-j)+\left(v_{p}(n!)-\sum_{k=1}^{n}v_{p}(k!^{m_{k}}m_{k}!)\right)+\left(e(n)-\sum_{k=1}^{n}m_{k}e(k)\right),
\end{align*}
where $e(n)$ is given by (\ref{eq3.6}). Therefore, we have
\begin{align*}
0 & =(p-1)\left(v_{p}(n!)-\sum_{k=1}^{n}v_{p}(k!^{m_{k}}m_{k}!)\right)\\
 & =n-s_{p}(n)-\sum_{k=1}^{n}\left(m_{k}(k-s_{p}(k))+m_{k}-s_{p}(m_{k})\right)\text{ by Lemma \ref{lem2.2},}\\
 & =\sum_{k=1}^{n}\left(m_{k}(s_{p}(k)-1)+s_{p}(m_{k})\right)-1\text{ since }n\text{ is a power of }p.
\end{align*}
It follows that there is exactly one $m_{k_{0}}\neq0$ and $n=k_{0}m_{k_{0}}$.
If $m_{n}=0$, then $m_{k_{0}}\geq p$ and
\[
\sum_{j=2}^{\sum_{k=1}^{n}m_{k}}v_{p}(nd-j)\geq v_{p}(nd-m_{k_{0}})>0.
\]
This is a contradiction, from which we conclude that $v_{p}(a_{n}')=v_{p}(a_{n})$
if $n$ is a power of $p$. Thus the first assertion is proved. For
the second assertion, we note that
\[
\frac{\varphi_{c}(z)-\varphi_{c}(w)}{z-w}=1-\sum_{n=1}^{\infty}\sum_{i=1}^{nd-1}\frac{a_{n}}{z^{i}w^{nd-i}}.
\]
If $z,w\in D(\infty,r_{N}^{1/d})$, then by Proposition \ref{prop3.4},
we have
\[
\left|\frac{a_{n}}{z^{i}w^{nd-i}}\right|_{p}<\frac{|a_{n}|_{p}}{r_{N}^{n}}=\left(\frac{p^{-v_{p}(a_{n})/n}}{\lim_{n\to\infty}p^{-v_{p}(a_{n})/n}}\right)^{n}<1.
\]
Therefore, $|\varphi_{c}(z)-\varphi_{c}(w)|_{p}=|z-w|_{p}$ on $D(\infty,r_{N}^{1/d})$.
\end{proof}

\section{\label{sec4} Proof of Theorem \ref{thm1.3}}

In this section, we focus on the properties of $b_{n}$ and give the
proof of Theorem \ref{thm1.3}. In addition to Proposition \ref{prop3.5},
we need two more inequalities.
\begin{prop}
\label{prop4.1} Let $p$, $N$, $d$, and $c$ satisfy Condition
\ref{condA} or \ref{condB}. Then
\begin{enumerate}
\item If $d\mid n$, then $v_{p}(da_{n})\leq v_{p}(a_{n/d})$.
\item If $1\leq i<n/d$, then $v_{p}(da_{n})<v_{p}(a_{i}c^{n-id})$.
\end{enumerate}
\end{prop}

\begin{proof}
If $d\mid n$, we let
\[
n/d=\sum_{k=0}^{N}m_{k}d^{k}\qquad\text{and}\qquad n=\sum_{k=0}^{N-2}m_{k}d^{k+1}+(m_{N-1}+m_{N}d)d^{N}
\]
be the canonical partitions (\ref{eq3.3}) of $n/d$ and $n$. Then
we have
\begin{align*}
v_{p}(da_{n}) & =v_{p}(d)+\sum_{k=0}^{N-2}v_{p}\left(\frac{a_{d^{k+1}}^{m_{k}}}{m_{k}!}\right)+v_{p}\left(\frac{a_{d^{N}}^{m_{N-1}+m_{N}d}}{(m_{N-1}+m_{N}d)!}\right)\text{ by Proposition \ref{prop3.3},}\\
 & =v_{p}(d)+\sum_{k=0}^{N-2}v_{p}\left(\frac{a_{d^{k}}^{m_{k}}}{d^{m_{k}}m_{k}!}\right)+v_{p}\left(\frac{a_{d^{N-1}}^{m_{N-1}}a_{d^{N}}^{m_{N}d}}{d^{m_{N-1}}(m_{N-1}+m_{N}d)!}\right)\text{ by Proposition \ref{prop3.3},}\\
 & =v_{p}(d)+\sum_{k=0}^{N}v_{p}\left(\frac{a_{d^{k}}^{m_{k}}}{d^{m_{k}}m_{k}!}\right)+m_{N}v_{p}(a_{d^{N}}^{d-1}d)-v_{p}\left(\frac{(m_{N-1}+m_{N}d)!}{m_{N-1}!m_{N}!}\right)\\
 & \leq v_{p}(d)+\sum_{k=0}^{N}v_{p}\left(\frac{a_{d^{k}}^{m_{k}}}{d^{m_{k}}m_{k}!}\right)+m_{N}v_{p}(a_{d^{N}}^{d-1}d)-m_{N}v_{p}(d!)\text{ by Lemmas \ref{lem2.1} and \ref{lem2.4},}\\
 & =v_{p}(a_{n/d})+\left(1-\sum_{k=0}^{N}m_{k}\right)v_{p}(d)+m_{N}\left((d-1)v_{p}\left(\frac{c}{d^{N+1}}\right)-v_{p}((d-1)!)\right)\\
 & \leq v_{p}(a_{n/d})\text{ by \eqref{eq1.5} and the RHS of \eqref{eq1.6}.}
\end{align*}
If $1\leq i<n/d$, then
\begin{align*}
v_{p}(a_{i}c^{n-id}) & \geq v_{p}(da_{id}c^{n-id})\text{ by part (1),}\\
 & =v_{p}(d)+v_{p}\left(\frac{c^{n}}{d^{id}(id)!}\right)-e(id)\text{ by Proposition \ref{prop3.3} and \eqref{eq3.6},}\\
 & >v_{p}(d)+v_{p}\left(\frac{c^{n}}{d^{n}n!}\right)-e(n)\text{ since }id<n,\\
 & =v_{p}(da_{n})\text{ by Proposition \ref{prop3.3}.}\qedhere
\end{align*}
\end{proof}
As mentioned in the introduction, we can consider (\ref{eq1.4}) as
a perturbation of the simpler equation (\ref{eq1.3}). Now we show
that the perturbation is insignificant in the following sense.
\begin{prop}
\label{prop4.2} Let $p$, $N$, $d$, and $c$ satisfy Condition
\ref{condA} or \ref{condB}. Then $v_{p}(b_{n})=v_{p}(a_{n})$ for
any $n\geq1$.
\end{prop}

\begin{proof}
We use induction to show that $v_{p}(a_{n}-b_{n})>v_{p}(a_{n})$,
which will then imply $v_{p}(b_{n})=v_{p}(a_{n})$. Considering the
degree $n$ terms of (\ref{eq1.3}) and (\ref{eq1.4}), we have
\[
da_{n}+\sum_{\substack{\sum_{k=0}^{n-1}m_{k}=d\\
\sum_{k=0}^{n-1}km_{k}=n
}
}{d \choose m_{0},m_{1},\dots,m_{n-1}}\prod_{k=1}^{n-1}a_{k}^{m_{k}}=\begin{cases}
a_{n/d}, & \text{if }d\mid n,\\
0, & \text{if }d\nmid n,
\end{cases}
\]
and
\begin{equation}
db_{n}+\sum_{\substack{\sum_{k=0}^{n-1}m_{k}=d\\
\sum_{k=0}^{n-1}km_{k}=n
}
}{d \choose m_{0},m_{1},\dots,m_{n-1}}\prod_{k=1}^{n-1}b_{k}^{m_{k}}=\sum_{id\leq n}q(n,i)b_{i}c^{n-id},\label{eq4.1}
\end{equation}
where $q(n,i)\in\mathbb{Z}$ and $q(n,n/d)=1$ if $d\mid n$. Therefore,
\begin{align*}
d(a_{n}-b_{n}) & +\sum_{\substack{\sum_{k=0}^{n-1}m_{k}=d\\
\sum_{k=0}^{n-1}km_{k}=n
}
}{d \choose m_{0},m_{1},\dots,m_{n-1}}\left(\prod_{k=1}^{n-1}a_{k}^{m_{k}}-\prod_{k=1}^{n-1}b_{k}^{m_{k}}\right)\\
 & =\begin{cases}
a_{n/d}-b_{n/d}-\sum_{id<n}q(n,i)b_{i}c^{n-id}, & \text{if }d\mid n,\\
-\sum_{id<n}q(n,i)b_{i}c^{n-id}, & \text{if }d\nmid n.
\end{cases}
\end{align*}
By the induction hypothesis and Proposition \ref{prop4.1}, we have
\[
v_{p}(a_{n/d}-b_{n/d})>v_{p}(a_{n/d})\geq v_{p}(da_{n})
\]
and
\begin{equation}
v_{p}(q(n,i)b_{i}c^{n-id})\geq v_{p}(a_{i}c^{n-id})>v_{p}(da_{n}).\label{eq4.2}
\end{equation}
By the induction hypothesis and Proposition \ref{prop3.5}, we have
\begin{align*}
v_{p} & \left({d \choose m_{0},m_{1},\dots,m_{n-1}}\left(\prod_{k=1}^{n-1}a_{k}^{m_{k}}-\prod_{k=1}^{n-1}b_{k}^{m_{k}}\right)\right)\\
 & =v_{p}\left({d \choose m_{0},m_{1},\dots,m_{n-1}}\left(\prod_{k=1}^{n-1}a_{k}^{m_{k}}-\prod_{k=1}^{n-1}(a_{k}-(a_{k}-b_{k}))^{m_{k}}\right)\right)\\
 & >v_{p}\left({d \choose m_{0},m_{1},\dots,m_{n-1}}\prod_{k=1}^{n-1}a_{k}^{m_{k}}\right)=v_{p}\left(\frac{d!}{m_{0}!}\prod_{k=1}^{n-1}\frac{a_{k}^{m_{k}}}{m_{k}!}\right)\geq v_{p}(da_{n}).
\end{align*}
Combining these inequalities together, we get $v_{p}(a_{n}-b_{n})>v_{p}(a_{n})$
and $v_{p}(b_{n})=v_{p}(a_{n})$.
\end{proof}
A consequence of Proposition \ref{prop4.2} is that Propositions \ref{prop3.3},
\ref{prop3.4}, and \ref{prop3.5} remain true if we replace $a_{n}$
by $b_{n}$. Therefore, the proof of Theorem \ref{thm1.3} is essentially
the same as the proof of Theorem \ref{thm1.2}.

\section{\label{sec5} Proof of Theorem \ref{thm1.7}}

In this section, we give the proof of Theorem \ref{thm1.7}.

If $v_{p}(c_{1})\geq0$ and $\Phi_{c_{1},c_{2}}:B(c_{1})=D(\infty,1)\to B(c_{2})$
exists, then $\Phi_{c_{1},c_{2}}$ must be of the form
\[
\Phi_{c_{1},c_{2},\omega}(z)=\omega z\left(1+\sum_{n=1}^{\infty}\frac{t_{n}}{z^{nd}}\right)
\]
for some $\omega$ with $\omega^{d-1}=1$. Let $x=z^{-d}$, then (\ref{eq1.7})
can be simplified as
\begin{align*}
\left(1+\sum_{n=1}^{\infty}t_{n}x^{n}\right)^{d} & =1+(\omega^{-1}c_{2}-c_{1})x+\sum_{n=1}^{\infty}\frac{t_{n}x^{nd}}{(1-c_{1}x)^{nd-1}}\\
 & =1+(\omega^{-1}c_{2}-c_{1})x+\sum_{n=d}^{\infty}\sum_{id\leq n}q'(n,i)t_{i}c_{1}^{n-id}x^{n},
\end{align*}
where $q'(n,i)\in\mathbb{Z}$ and $q'(n,n/d)=1$ if $d\mid n$. We
can imitate the proof of Proposition \ref{prop4.2} to prove the following
proposition.
\begin{prop}
\label{prop5.1} Let $p$, $N$, $d$, and $c=\omega^{-1}c_{2}-c_{1}$
satisfy Condition \ref{condA} or \ref{condB}. Let $c_{1}$ satisfy
$v_{p}(c_{1})\geq0$ and $v_{p}(c_{1})\geq v_{p}(c)$, then
\begin{enumerate}
\item We have $v_{p}(t_{n})=v_{p}(a_{n})$ for any $n\geq1$.
\item The maximal convergent open disks of $\Phi_{c_{1},c_{2},\omega}(z)$
and $\Phi_{c_{1},c_{2},\omega}^{-1}(z)$ are both $D(\infty,r_{N}^{1/d})$.
Moreover, $\Phi_{c_{1},c_{2},\omega}(z)$ gives a bijective isometry
from $D(\infty,r_{N}^{1/d})$ onto itself.
\item $\Phi_{c_{1},c_{2},\omega}(z)$ does not give an analytic conjugacy
between $B(c_{1})$ and $B(c_{2})$.
\end{enumerate}
\end{prop}

\begin{proof}
The proof of part (1) is essentially the same as the proof of Proposition
\ref{prop4.2}, except that we need to replace $b_{n}$ by $t_{n}$,
replace (\ref{eq4.1}) by
\[
dt_{n}+\sum_{\substack{\sum_{k=0}^{n-1}m_{k}=d\\
\sum_{k=0}^{n-1}km_{k}=n
}
}{d \choose m_{0},m_{1},\dots,m_{n-1}}\prod_{k=1}^{n-1}t_{k}^{m_{k}}=\sum_{id\leq n}q'(n,i)t_{i}c_{1}^{n-id},
\]
and replace (\ref{eq4.2}) by
\[
v_{p}(q'(n,i)t_{i}c_{1}^{n-id})\geq v_{p}(a_{i}c^{n-id})>v_{p}(da_{n}).
\]
The proof of part (2) is essentially the same as the proof of Theorem
\ref{thm1.2}. Part (3) follows because $D(\infty,r_{N}^{1/d})$ is
strictly smaller than $B(c_{1})=D(\infty,1)$.
\end{proof}
Now we are ready to give the proof of Theorem \ref{thm1.7}.
\begin{proof}
[Proof of Theorem \ref{thm1.7}] By (\ref{eq1.8}), we have $v_{p}(c_{1})\geq v_{p}(c_{2})=v_{p}(\omega^{-1}c_{2}-c_{1})$
for any $\omega$ with $\omega^{d-1}=1$. By Proposition \ref{prop5.1},
none of $\Phi_{c_{1},c_{2},\omega}(z)$ gives an analytic conjugacy
between $B(c_{1})$ and $B(c_{2})$.
\end{proof}

\end{document}